\documentclass[reqno,12pt]{amsart}

\usepackage[utf8]{inputenc}

\usepackage{enumerate}
\usepackage[margin=1in,marginparwidth=0.7in]{geometry}
\usepackage{ifpdf}
\usepackage{amsmath}
\usepackage{amsfonts}
\usepackage{amssymb}
\usepackage{amsthm}
\usepackage[ocgcolorlinks,hyperfootnotes=false,colorlinks=true,citecolor=blue,linkcolor=blue,urlcolor=blue]{hyperref}
\usepackage{setspace}
\usepackage{amsrefs}
\usepackage{nicefrac}
\usepackage{graphicx}
\usepackage{color}
\usepackage{mathtools}

\newcommand{\ignore}[1]{}

\newcommand{\abs}[1]{\left\lvert {#1} \right\rvert}
\newcommand{\sabs}[1]{\lvert {#1} \rvert}

\newcommand{\C}{{\mathbb{C}}}
\newcommand{\R}{{\mathbb{R}}}

\newcommand{\sF}{{\mathcal{F}}}

\newcommand{\sI}{{\mathcal{I}}}

\newcommand{\sX}{{\mathcal{X}}}

\newcommand{\ourgeneric}[1]{{${\mathcal{O}}_{#1}$-generic}}

\newcommand{\rank}{\operatorname{rank}}

\newtheorem{thm}{Theorem}[section]

\newtheorem{prop}[thm]{Proposition}

\newtheorem{lemma}[thm]{Lemma}

\theoremstyle{definition}
\newtheorem{defn}[thm]{Definition}
\newtheorem{example}[thm]{Example}

\theoremstyle{remark}
\newtheorem{remark}[thm]{Remark}

\author{Bernhard Lamel}
\address{Faculty of Mathematics, University of Vienna, Austria}
\email{bernhard.lamel@univie.ac.at}
\author{Ji\v{r}\'{\i} Lebl}
\thanks{The second author was in part supported by Simons Foundation collaboration grant 710294.}
\address{Department of Mathematics, Oklahoma State University,
Stillwater, OK 74078, USA}
\email{lebl@okstate.edu}

\date{September 19, 2024}

\title{Intrinsic complexification of real-analytic varieties}


\begin{document}


\begin{abstract}
We introduce a framework for Segre varieties for singular real-analytic 
subvarieties of a complex space and utilize it to study the intrinsic 
complexifications of these subvarieties. Many examples illustrate the subtle issues
arising in the singular setting. 
\end{abstract}

\maketitle



\section{Introduction}

In this paper, we study one of the basic objects of real-analytic CR geometry, the {\em Segre
variety}, for possibly singular real-analytic subvarieties of a complex space. For 
regular real-analytic submanifolds which are furthermore CR, these have been studied and
used extensively over the past decades; instead of referencing the many authors (and we are bound to forget some!) who have used this tool, we point the interested reader to the book by Baouendi, Ebenfelt, and Rotschild \cite{BER:book} and
the many references in there. Segre variety techniques are also important 
in the singular setting, implicitly appearing for example in Diederich and Forn{\ae}ss' proof of their Theorem 4 in \cite{MR0477153}. One of the first explicit uses of Segre
varieties in the singular setting 
is probably in the paper of Burns and Gong~\cite{MR1704996}.  D'Angelo in his book~\cite{MR1224231} 
also introduces ideas extending to the singular setting.
A key complication is that in the singular setting, Segre varieties a priori 
depend on the defining functions, the neighbourhoods on which we work, and possible noncoherence of 
the structure sheaf implies that any definition might even depend on the point that we are 
working at. 

Our first goal here is to introduce Segre varieties in the singular setting carefully, illustrating
by many examples along the way possible pitfalls for naive adaptations from the regular setting. Our 
main motivation was to explore concepts from CR geometry which in the regular setting are 
amenable to description via Segre varieties in this singular setting. This is done in 
section~\ref{sec:varieties} and leads us to introduce 
the notion of {\em Segre nondegeneracy}, extending our earlier work \cite{MR4311600} for 
singular totally real varieties to the general situation.  

In particular, we are able shed light on 
the concept of {\em intrinsic complexifications}, the smallest complex-analytic subvariety containing a given real-analytic subvariety,  and the related notion of {\em genericity} for singular real-analytic 
subvarieties in this framework, and show how the behavior of the Segre varieties near a singular point influences 
the behaviour of the intrinsic complexification. While examples show that the dimension of 
the intrinsic complexification depends in general on the point on the variety, in section~\ref{sec:intrinsic} we are able to 
give a number of sufficient conditions ensuring that the intrinsic complexification is well-behaved. 

\section{Varieties and complexifications}\label{sec:varieties}

By a real-analytic subvariety $X$ of an open set $U \subset \R^m$
we mean a closed subset of $U$ that is locally given by the vanishing of a
family of real-analytic functions.
Denote by $X_{\rm{reg}}$ the set of regular points, that is, points at which
$X$ is a real-analytic submanifold. Recall that the 
dimension of $X$ is just the maximum of the 
dimensions of the regular points of $X$ and the dimension of
$X$ at $p$ (or  the dimension of the germ of $X$ at $p$)
is the infimum over all representatives of the germ $(X,p)$. 

Let $(X,q) \subset (\C^n,q)$ be a germ of a real-analytic subvariety.
Suppose $(X,q)$ is of (real) codimension $k$ at the origin,
that is, dimension $2n-k$.
Given a set $U \subset \C^n$, denote by $U^* = \{ z : \bar{z} \in U \}$ its complex conjugate. We are 
going to identify $\C^n$ with the diagonal $\iota (\C^n) \subset \C^{2n}$, where $\iota (z) = (z, \bar z)$. In particular, we can consider $U\subset \C^{2n} $ by identifying it with $\iota(U) = \{ (z, \bar z) \in U\times U^* \colon z \in U \}$.

\begin{defn}
For a germ $(X,q)$ of a real-analytic subvariety,
the smallest germ $(\sX,(q,\bar{q})) \subset (\C^n \times \C^n, (q,\bar{q}))$ of
a complex analytic subvariety
such that $\iota((X,q)) \subset (\sX,(q,\bar{q}))$ is called the
\emph{complexification} of $(X,q)$.

Let $(X,q)$ be a germ of a real-analytic subvariety and
$U \subset \C^n$ a neighborhood of the origin.
Let $X^U \subset U$ be the smallest 
real-analytic subvariety containing the germ $(X,q)$.
We let $\sX^U$ be the smallest
complex-analytic subvariety of $U \times U^*$ such that $\iota(X^U) \subset \sX^U$.
We call $\sX^U$ the \emph{precomplexification} of $X^U$ if
$\iota(X^U) = \sX^U \cap \iota(\C^n)$.

We say $\sX^U$ is the \emph{complexification} of $X$ if
$(\sX^U,(p,\bar{p}))$ is the complexification
of $(X,p)$ for every $p \in X^U$.

The \emph{Segre variety} (with respect to $U$) of a point $p\in U$ is defined as
\begin{equation}
\Sigma_p^U = \{ z \in U \colon (z,\bar p) \in \sX^U \}.
\end{equation}
We say $X^U$ is \emph{Segre nondegenerate} if $\dim \Sigma_p^U$ is constant
for all $p \in U$.
We say $X^U$ is \emph{maximally Segre nondegenerate} if $\dim \Sigma_p^U = n-k$
for all $p$, where $k$ is the real codimension of $X^U$.
The germ $(X,q)$ is (maximally) Segre nondegenerate
if there is some (maximally) Segre nondegenerate representative.
\end{defn}

\begin{example}\label{ex:cartanumbrella}
    If $X$ is not coherent at $q$, for example if $X\subset \C^2_{(z,w)}$ is the 
    variant of the Cartan umbrella
    given by 
    \[ (x^2 + y^2) s = x^3, \quad z=x+iy, \quad w=s+it, \]
    then its precomplexification is not a complexification, 
    because near every point of $\C_w \subset \C^2$ but 
    at the origin  the variety $X$ is $2$-dimensional (it agrees 
    with $z=0$ there) while at the origin, 
    it is $3$-dimensional. The precomplexification therefore 
    has also dimension $3$ throughout any neighbourhood, 
    which is larger than the dimension of the complexification 
    of $X$ at the points just mentioned. 
\end{example}

\begin{example}
    A difference in dimension is not the only reason which can make a precomplexification fail to be 
    an actual complexification: For an example 
    where the precomplexification is 
    of the same dimension but develops a 
    singularity near regular points of the 
    variety, we refer the reader to Example 2.5. in \cite{MR4407043}. 
\end{example}

\begin{example}
    Real hypersurfaces are always 
    maximally Segre nondegenerate at any 
    regular point, but can fail to be 
    Segre nondegenerate at singular points.
    For example, the subvariety given by $\abs{z}^2=\abs{w}^2$ in $\C^2$
is Segre degenerate at the origin
as $\Sigma_0^U = U$ for any connected neighborhood
$U$ of the origin, while $\Sigma_p^U$ is a complex line for any $p \not= 0$.
We collect some classical results in a proposition.
\end{example}

\begin{example} One example of varieties which are Segre degenerate comes from some types of
CR-singular real-analytic submanifolds of $\C^n$, for example $w = |z|^2$ in $\C^2$. This variety is defined by the equations $w= z \bar z$ and $\bar w = z\bar z$. The Segre variety of a 
point $p=(\bar \alpha, \bar \beta)$ is given by $\Sigma_{(\bar \alpha,\bar \beta)}^{\C^2} = \{ (\frac{\beta}{\alpha}, \beta) \}  $ if $\alpha\neq 0$, and $\Sigma_{(0,0)}^{\C^2} = \{ (z,0) \colon z \in \C \} $.

   On the other hand, the CR-singular
   submanifold defined by $w= z^2 + \bar z^2$ is maximally Segre nondegenerate.
\end{example}

For simplicity, let us assume for the 
rest of the paper that our variety passes through the origin. The following proposition collects some simple
standard facts.  For more details, see for example~\cite{MR1013362}.

\begin{prop} \label{prop:goodnbhd}
Let $(X,0)\subset (\C^n,0)$ be the germ of a real-analytic subvariety.  Then there exists a connected neighborhood $U\subset \C^n$ of the origin and a
real-analytic function $\rho \colon U \to \R^\ell$ such that:
\begin{enumerate}[\rm (i)]
\item
The power series $\rho(z,\bar{z})$ at the origin converges on $U \times U^*$
as a power series $\rho(z,\xi)$.
\item
$X^U = \{ z : \rho(z,\bar{z}) = 0 \}$.
\item
$\sX^U = \{ (z,\xi) : \rho(z,\xi) = 0 \}$.
\item
The components of $\rho$ generate the ideal 
$I_0(X)$.
\item
$(X^U,0) = (X,0)$.
\item
The (complexified) components of $\rho$ generate the ideal 
$\sI_{(p,q)}(\sX^U)$ for every $(p,q) \in \sX^U$.
\item
In particular, $(\sX^U,0)$ is the complexification of $(X,0)$.  Note, however, that $(\sX^U,(p,\bar{p}))$ may fail to be the complexification of $(X^U,p)$ for $p$ different from the origin.
\item
$\dim_{\C} \sX^U = \dim_{\R} X^U = \dim_{\R} (X,0)$.
\item
$\iota(X^U) = \sX^U \cap \iota(\C^n)$,
in other words $\sX^U$ is the precomplexification of $X^U$.
\item
If $(X,0)$ is irreducible, then $X^U$ and $\sX^U$ are irreducible.
\item
The irreducible components of $(X,0)$ are in one-to-one correspondence with the irreducible components of $(\sX^U,0)$.
\end{enumerate}
Furthermore, $U \subset \C^n$ is a neighborhood satisfying the conditions above
and $W \subset U$ is another neighborhood of
the origin, then there exists a neighborhood $U_1 \subset W$,
with the properties above (using the same $\rho$).
The neighborhood $U_1$ can always be taken to be a ball or a polydisc.
\end{prop}

The proof is to take the generator of $I_0(X)$ and find a small enough neighborhood
$U$ where the power series converges as required.  Then, as subvarieties have
locally the structure of cones, we can find $U$ small enough so that the
irreducible components correspond as given.  It is however not enough to
take just a small enough neighborhood.  In particular, a clearly necessary
condition is for the second Cousin problem to be solvable in $U \times U^*$,
otherwise $\rho$ may vanish on a larger set than $\sX^U$.
Using a ball or a polydisc is sufficient.

\begin{defn}
Given $(X,0)$, we say a neighborhood $U$ is \emph{good} (or
\emph{good for $(X,0)$}) if it satisfies the
conclusions of the Proposition~\ref{prop:goodnbhd}.
With a good neighbourhood $U$ of the origin and $\rho$ as in Proposition~\ref{prop:goodnbhd},
we define the Segre variety of $(X,0)$   by
\begin{equation}
\Sigma_p^U = \{ z \in U : \rho(z,\bar{p}) = 0 \} .
\end{equation}
Furthermore, the proposition implies that the germ of the Segre variety
at the origin $(\Sigma_0^U,0)$ is well-defined.
Write $\Sigma_0$ for this germ, which we call the \emph{Segre variety germ} at $0$.
\end{defn}
However, note that $\Sigma_p^U$ depends on $U$ and also on the point $0$.
For $p\not=0$, the germ $(\Sigma_p^U,p)$ may be very 
different from the germ of the Segre variety we get
if we replace the origin with $p$ and consider neighborhoods $U$
of $p$ which do not contain $0$.  In other words, the precomplexification and
Segre varieties are always with respect to a certain fixed point, in our case, the
origin.

\begin{prop}
Let $(X,p)$ be a germ of a real-analytic subvariety, then
there exists a good neighborhood $U$ for $(X,p)$ such that
the irreducible components of $(\Sigma_p^U,p)$
are in one-to-one correspondence with irreducible components
of $\Sigma_p^U$.
\end{prop}

\begin{proof}
Suppose $U_1$ is a good neighborhood of $p$.
Suppose that two irreducible components of the germ $(\Sigma_p^U,p)$
correspond to the same component $S$ of $\Sigma_p^U$.  Then there is some
smaller neighborhood $U_2 \subset U$ of $p$ so that the components of
$S \cap U_2$
and the germ
$(S,p)$ are in one-to-one correspondence.
By making $U_2$ possibly smaller,
we can make $U_2$ to be a good neighborhood for $X$ at $p$.
The irreducible components of $S \cap U_2$ must be
then irreducible components of $\Sigma_p^{U_2}$.
As $(\Sigma_p^{U_1},p) = (\Sigma_p^{U_2},p)$, the two components
that corresponded to the same component $S$ do not correspond to two
distinct components of $S \cap U_2$ and therefore of $\Sigma_p^{U_2}$.
\end{proof}

Simply having a good neighborhood may not be enough to get the conclusion of
the proposition above.  While for a good neighborhood $U$ for $(X,p)$,
the components of $(X,p)$ and $X^U$ are in one-to-one correspondence,
one may need to look at an even smaller neighborhood to get correspondence
of the components of $\Sigma_p^U$ and $(\Sigma_p^U,p)$.

\begin{example}
First consider the example $X \subset \C^2$ given by
\begin{equation}
(z_1-1)z_1^2 + z_2^2 + \overline{(z_1-1)z_1^2 + z_2^2} = 0 .
\end{equation}
The set $U = \C^2$ (or any ball or polydisc centered at the origin)
is a good neighborhood for $X$.
Near the origin, we can do a biholomorphic change of variables in $z_1$
to transform $(z_1-1)z_1^2$ to just $z_1^2$, and so after this local
biholomorphism, the
variety is given by
\begin{equation}
z_1^2 + z_2^2 + \bar{z}_1^2 + \bar{z}_2^2 = 0 ,
\end{equation}
which is irreducible, and hence $(X,0)$ is irreducible.  However,
the Segre variety at the origin is given by
\begin{equation}
(z_1-1)z_1^2 + z_2^2  = 0 ,
\end{equation}
which is irreducible as a subvariety of $\C^2$
but it 
is locally reducible at the origin.  Near the origin,
it is a union of two nonsingular
curves meeting at the origin.
So if we take a small enough ball around the origin, the
Segre variety with respect to that ball will have two components
as $(\Sigma_p^U,0)$ does.

The example can be modified to show that we cannot always pick a good neighborhood
of a point $0$ that will give us the desired property for all points $p$
near $0$.  Simply bihomogenize the equation to create a complex
cone in $\C^3$:
\begin{equation}
(z_1-\bar{z}_3)z_1^2 + \bar{z}_3 z_2^2 + (\bar{z}_1-z_3)\bar{z}_1^2 +
z_3\bar{z}_2^2 = 0 .
\end{equation}
Then no neighborhood $U$ of $0=(0,0,0)$ satisfies
the conclusion of the proposition for points
$p=(0,0,\epsilon)$ for the same reasons as above.
We would have to pick a smaller neighborhood of those points.
\end{example}

For so-called coherent subvarieties, we can find a $\rho$ that defines
the ideal at every point and hence, among many other things,
the Segre varieties do not really depend on the point.

\begin{defn}
A real-analytic subvariety $X \subset U$ is \emph{coherent} if
the sheaf of germs of real-analytic functions vanishing
on $X$ is a coherent sheaf.
We say the germ $(X,0)$ is coherent if there exists a representative $X$
that is coherent.
\end{defn}

Equivalently, $X$ is coherent if there
exists a neighborhood $W$ of $\iota(X)$ and a complex subvariety
$\sX \subset W$ which is the complexification of $X$, that is,
$\iota(X) = \sX \cap \iota(\C^n)$ and
$(\sX,(p,\bar{p}))$ is the complexification of $(X,p)$ for every $p \in X$.

Note that the neighborhood $W$ need not have product structure.  We may need
to take $X$ inside a smaller neighborhood.  That is, near a point,
we can always pick a good $U$ so that $U \times U^* \subset W$.

\begin{prop}
Let $(X,0)$ be a germ of a coherent real-analytic subvariety.
Then there exists a sufficiently small good neighborhood $U$ at the
origin, such that $U$ is also a good neighborhood for $X^U$ at all $p \in X^U$.
\end{prop}

We also note that the set of regular points of top dimension 
is dense in an irreducible coherent subvariety; the Cartan umbrella from above shows this is not necessarily true 
for noncoherent varieties.

The following proposition may seem trivial, but it highlights the role that coherence plays for some CR properties. Note that in the real analytic setting, the set of regular points might very well not be connected, as just discussed. 

\begin{prop}\label{pro:CRdimconstant}
    Let $X$ be a coherent irreducible real-analytic subvariety of an open set $U\subset \C^n$. Let $Y\subset X_{\rm reg}$ be the subset of CR points of $X$. Then the CR dimension of $Y$ is constant. 
\end{prop}

\begin{proof}
    Let $p\in X$. Since $X $ is coherent, we can choose real-analytic functions $\rho_1, 
    \dots,\rho_\ell$ near $p$ which generate $I_q (X)$ for all $q\in V \subset X$, where $V$ is an open neighbourhood of $p$ in $X$. At a CR point $q$, 
    the CR dimension $k_q$ is given by $n-\rank \bar \partial \rho $. Now the maximum rank of $\bar\partial \rho$ is
    attained on an open, dense subset of $V$ and therefore the CR dimension needs to remain constant for CR points 
    throughout $V$. 
\end{proof}

\begin{example}
    The following example shows how the proposition fails to work in the non-coherent setting by adapting 
    Example \ref{ex:cartanumbrella}. We consider 
    the subvariety $X\subset \C^2_{(z,w)}$ given by 
    \[ (t^2 + y^2) s= t^3, \quad z=x+iy , \quad w = s+it. \]
    The umbrella handle is totally real (given by $t=y=0$) while the other regular points are of CR dimension $1$.     
\end{example}

\section{Intrinsic complexification}\label{sec:intrinsic}

We generalize the definitions of ``intrinsic complexification'' and ``generic''
to subvarieties.

\begin{defn}
Suppose $X \subset U\subset \C^n$ is a real-analytic subvariety.  Let $Y \subset U$
be the smallest complex-analytic subvariety containing $X$.
The subvariety $Y$ is called the \emph{intrinsic complexification} of $X$ (in $U$).

Suppose $(X,p) \subset (\C^n,p)$ is a germ of a real-analytic subvariety.
Let $(Y,p)$ be the smallest germ of a complex subvariety of $(\C^n,p)$
such that $(X,p) \subset (Y,p)$.
Then $(Y,p)$ is called the \emph{intrinsic complexification} of $(X,p)$.
If $(Y,p) = (\C^n,p)$, then we say $(X,p)$ is \ourgeneric{p}.
\end{defn}

The reader might wonder why we introduce a new notion of 
genericity here. The examples below will discuss the relation
of the dimension of the complexification with 
classical genericity and show that we cannot use the 
same terminology without inconsistencies.

\begin{prop}
If $X \subset U$ is a real-analytic subvariety and $p \in X$,
then there exists a neighborhood $W$ of $p$ which is good for $(X,p)$
and such that if $Y \subset W$ is the intrinsic complexification in $W$ of 
$X \cap W$, then $(Y,p)$ is the intrinsic complexification of $(X,p)$.
\end{prop}

\begin{proof}
Simply start with the germ of the intrinsic complexification,
and take a representative.  The neighborhood $W$ in which
this representative
is defined can be taken to be small enough to be good for $(X,p)$.
If we also take it small enough so that $Y$ has the same number
of components as $(Y,p)$, we find that $X \cap W \subset Y$ since
$(X,p) \subset (Y,p)$.
\end{proof}

\begin{example} A (germ of a) real submanifold $M$ which is generic in the
classical sense, i.e.\ a CR manifold
whose CR dimension plus codimension is equal to $n$, is \ourgeneric{p}.
Otherwise, the intrinsic complexification is the smallest complex submanifold 
containing $M$. For CR singular submanifolds the intrinsic complexification need no longer
be nonsingular.  For example, the submanifold given by
\begin{equation*}
w_1 = {\abs{z}}^4, \quad w_2 = {\abs{z}}^6
\end{equation*}
has $w_1^3=w_2^2$ as the intrinsic complexification, which is singular. 
\end{example}
\begin{example}
Note that while for a CR submanifold,
being generic at a point is simply a statement
about the tangent space, the same is not true for
a subvariety (using perhaps tangent cones)
or a CR singular submanifold.
The tangent space of $M$ given by $w=\sabs{z}^2$ at the origin 
is a complex line, and hence not a generic subspace of $T_0 \C^2$.
On the other hand, the intrinsic complexification of $(M,0)$
is $(\C^2,0)$ since at points near the origin, $M$ is a generic submanifold.
\end{example}

\begin{prop}
If $U \subset \C^n$ is a domain and
$X \subset U$ is a real-analytic subvariety.
Then
the dimension
of the intrinsic complexification of $(X,p)$ is an upper-semicontinuous function of $p$; in particular,
the set of not \ourgeneric{p} points is open in $X$.
\end{prop}

\begin{proof}
Take a point $p \in X$.  If the dimension of the intrinsic complexification
of $X$ at $p$ is $k$, then there exists a local complex variety $Y$
of dimension $k$ that contains $X$ near $p$.  But then $Y$ contains $X$
(as germs) at nearby points, so the dimension of the intrinsic
complexification at all points near $p$ is at most $k$.
\end{proof}

\begin{prop}
Suppose $(X,0) \subset (\C^n,0)$ is an irreducible germ of a maximally Segre
nondegenerate subvariety of real codimension $k$.
Let $U$ be a good neighborhood for $(X,0)$ and let $Y \subset X^U_{\rm{reg}}$ be the set
of CR regular points of dimension $2n-k$.
Then $Y$ is a generic submanifold, and
consequently $X^U$ is \ourgeneric{p}\ for every $p \in \overline{Y} \cap U$. In particular, if $X^U$ is in addition
coherent, then $X^U$ is \ourgeneric{p} for every
$p \in U$.
\end{prop}

The example \ref{ex:cartanumbrella} shows that $X_{\rm reg}^U$ may fail to be generic in an open subset, as the set where $z=0$, $w\neq 0$ is a complex manifold, while $X$ is maximally Segre nondegenerate.

\begin{proof}
Let $\rho = (\rho_1,\ldots,\rho_\ell)$ be the defining functions from Proposition~\ref{prop:goodnbhd}.
As these functions are the defining functions for $\sX^U$, they generate the
ideals $\sI_{(p,q)}(\sX^U)$ for all $(p,q) \in \sX^U$.  The (complex)
dimension of $\sX^U$ is $2n-k$, and hence the differentials of the functions
$\rho_1,\ldots,\rho_\ell$ span a $k$-dimensional set at all the regular points
of $\sX^U$.  The set $X^U$ cannot be contained in the singular set of
$\sX^U$ as $\sX^U$ is the smallest complex analytic subvariety of
$U \times U^*$ that contains $\iota(X^U)$.  That means that the set of points of
$X^U$ where the differentials of $\rho_1,\ldots,\rho_\ell$ have rank $k$
is $X^U \setminus S$ for some proper real subvariety $S \subset X^U$.  We
assumed that $(X,0)$ is irreducible, and so as $U$ is good for $(X,0)$, the
subvariety $X^U$ is irreducible.  Thus, $\dim S < \dim X^U$.  That is,
$\rho_1,\ldots,\rho_\ell$ have rank $k$ on
an open dense set of the regular points of dimension $2n-k$.
An open dense set of those points are points where $X^U$ is a
CR submanifold, denote this set by $Y'$.  Note that $Y'$ is open and dense
in $Y$.

As $\rho_1,\ldots,\rho_\ell$ generate the ideal $I_p(X^U)$
for all $p \in Y'$, we have that the germ
$(\Sigma^U_p,p)$ is the germ of the Segre variety of the germ $(Y',p)$.
As we are dealing with a CR submanifold, it is standard that as 
$\Sigma^U_p$ has dimension $n-k$, that $(Y',p)$ is a generic submanifold
near $p$.
As $Y'$ is open and dense in $Y$, we find that $Y$ is a generic submanifold
at all points.

If $X^U$ is in addition coherent, then
$X^U_{\rm{reg}}$ are all of dimension $2n-k$. In addition, Proposition~\ref{pro:CRdimconstant} shows that 
the CR dimension at points of $Y$ is constant,  and hence $X^U$ is \ourgeneric{p}
for every $p \in U$.
\end{proof}

\begin{example}
 Let $X \subset \C^2$ be the Cartan umbrella from Example~\ref{ex:cartanumbrella} defined by
$s(x^2+y^2) = x^3$, then even though incoherent, $X$ is maximally Segre nondegenerate.  
At the points of $X$ where $z\not=0$, the variety is a real-analytic submanifold of dimension 3 (hypersurface) 
and so its intrinsic complexification is $(\C^2,p)$, meaning it is two-dimensional.  However,
at points where $z=0$ and $s\not= 0$, the subvariety is simply the complex line
$z=0$, and hence this line is its intrinsic complexification at these points and thus
the intrinsic complexification at those points is of dimension 1.
Thus, for incoherent varieties, the dimension of the intrinsic complexification need not
be constant.
\end{example}

Let us prove that maximally Segre nondegenerate means that we can solve for the right
number of barred variables. 

\begin{prop}
Suppose $(X,0) \subset (\C^n,0)$ is an irreducible germ of
a maximally Segre nondegenerate subvariety of codimension $k$.
Then after perhaps a complex linear change of coordinates,
there is a good neighborhood
$U = U' \times U'' \subset \C^{n-k} \times \C^k$
and a representative $X^U \subset U$,
such that given any $z \in U$ and $\xi' \in U'$,
there exist finitely many points $\xi'' \in U''$
such that $(z,\xi',\xi'') \in \sX^U$.  In other words,
the projection
from $\sX^U$ to $\C^n \times \C^{n-k}$ is finite.
\end{prop}

\begin{proof}
Let $\pi \colon \C^n \to \C^{n-k}$ be the projection onto the first $n-k$
coordinates.
Do a linear change of coordinates so that $\pi|_{\Sigma_0^U}$ is a finite map,
that is, for a given small enough $U = U' \times U''$ given any $z' \in U'$
there are finitely many $z'' \in U''$ such that $(z',z'') \in \Sigma_0^U$.
Then $\Sigma_p^U$ has this property for any $p$ near $0$, and we can
assume that $U$ is small enough so that it is true for all $p \in U$.

The result follows by considering the
Segre variety $\Sigma_{\bar{z}}^U$  as $z$ varies as the fiber of the projection onto the first $n$ coordinates of $\mathcal{X}^U$.  Above each $z$, given $\xi'$
we look for the solution $\xi''$ in each Segre variety.
\end{proof}

The following example shows that the dimension of the intrinsic complexification 
at a CR singular point, even of a regular variety, is not governed by simple rank or dimension 
considerations and takes into account more subtle structure.

\begin{example} \label{example:nonalgex}
Let $M \subset \C^3$ in coordinates $(z,w_1,w_2)$ be given by
\begin{equation}
w_1 = z e^{z\bar{z}}, \qquad w_2 = z e^{\lambda z\bar{z}}
\end{equation}
for some irrational $\lambda$.  This $M$ is $2$-dimensional CR singular
submanifold of $\C^3$.  It is totally-real at points where 
$\frac{\partial}{\partial \bar{z}} \left[
z e^{z\bar{z}}
\right]$ and
$\frac{\partial}{\partial \bar{z}} \left[
z e^{\lambda z\bar{z}}
\right]$ are nonzero, meaning $M$ is CR except for the origin.  At the
origin, the CR dimension is 1.  At all points outside the origin, being a 2-dimensional
totally-real manifold means that the (germ of the) intrinsic complexification is 2
dimensional.  Let us show that the intrinsic complexification at the origin
is 3-dimensional, that is, $(\C^3,0)$.  We only need to show that any
germ of a holomorphic function vanishing on $M$ at the origin vanishes
identically.  So suppose
\begin{equation}
f(z,w_1,w_2) =
\sum_{d=0}^\infty
f_d(z,w_1,w_2)
=
\sum_{d=0}^\infty
\sum_{j+k+\ell=d}
c_{jk\ell} z^j w_1^k w_2^\ell
\end{equation}
is a holomorphic function that vanishes on $M$ expanded in homogeneous parts.
In other words,
\begin{equation}
0=
\sum_{d=0}^\infty
\sum_{j+k+\ell=d}
c_{jk\ell} z^j (z^k e^{k z\bar{z}}) (z^\ell e^{\ell \lambda z\bar{z}})
=
\sum_{d=0}^\infty
z^d
\sum_{j+k+\ell=d}
c_{jk\ell} e^{(k+\lambda \ell) z\bar{z}} .
\end{equation}
Writing $z= re^{i\theta}$, we find
\begin{equation}
0=
\sum_{d=0}^\infty
r^d e^{id\theta}
\sum_{j+k+\ell=d}
c_{jk\ell} e^{(k+\lambda \ell) r^2} .
\end{equation}
By uniqueness of Fourier series, we get that for every $d$,
\begin{equation}
0=
\sum_{j+k+\ell=d}
c_{jk\ell} e^{(k+\lambda \ell) r^2} .
\end{equation}
If $\lambda$ is irrational, then again by uniqueness of Fourier series, we
find that
$c_{jk\ell} = 0$ for all $j,k,\ell$.

So we have a nonalgebraic example where the intrinsic complexification is
not of constant dimension.  It is coherent as it is a submanifold.

Let us also note that generally, at the totally-real points, the Segre
variety is discrete (dimension 0).  However, at the origin, we plug in $0$
for the barred variable in
all 4 equations,
\begin{equation}
w_1 = z e^{z\bar{z}}, \quad w_2 = z e^{\lambda z\bar{z}}, \quad
\bar{w}_1 = \bar{z} e^{z\bar{z}}, \quad \bar{w}_2 = \bar{z} e^{\lambda z\bar{z}},
\end{equation}
to get
\begin{equation}
w_1 = z, \quad w_2 = z, \quad 0=0, \quad 0=0 .
\end{equation}
That is, the Segre variety at the origin, $\Sigma_0$, is one-dimensional.

We remark that if $\lambda$ is rational, then the intrinsic complexification at
the origin is 2-dimensional.  Therefore, the dimension of the intrinsic complexification depends on
$\lambda$ and not simply the dimension of $\Sigma_0$.
\end{example}

We now gather positive results for the dimension of the intrinsic complexification. 

\begin{lemma} \label{lemma:constsigma}
Suppose $X \subset U \subset \C^n$ is an
irreducible real-analytic subvariety, $U$ is a good neighborhood for $X$ at $p \in X$,
and $X$ is Segre nondegenerate,
then the intrinsic complexification of $X$ in $U$
is given by the projection of the complexification $\sX^U$ onto the first $n$ components.

Furthermore, the dimension of the intrinsic complexification at $p$ in $U$ is $\dim_\R X - \dim_\C \Sigma_p$.
\end{lemma}

\begin{proof}
    Denote by $\pi_1 \colon \C^{2n}  \to \C^n $ the 
      projection onto the first $n$ coordinates. Applying the local version of the semi-proper mapping theorem \cite[Chapter 7, Theorem 11E]{Whitney:book} we see that $\pi_1|_{\sX^U} (\sX^U)$ is a complex-analytic variety in $U$ of dimension $\dim_\C \sX^U - \dim_\C \Sigma_p^U  =\dim_\R X - \dim_\C \Sigma_p$.

    If $f\colon U \to \C$ is a holomorphic function vanishing on $X$, 
    its extension $f\circ \pi_1$ to $U\times U^*$ vanishes on $\sX^U$, 
    hence $f$ vanishes on $\pi_1 (\sX^U)$. Hence $\pi_1 (\sX^U)$ is the intrinsic complexification of $X$ at $p$ in $U$. 
\end{proof}

\begin{thm} \label{thm:constsigma}
If $U \subset \C^n$ is a domain and
$X \subset U$ is a coherent irreducible real-analytic subvariety where the
Segre variety germ $\Sigma_p$ is of constant dimension for $p \in X$, then the intrinsic complexification of $X$ at every point is of the same dimension.

Similarly, if near some point $X$ is a real-analytic CR
submanifold of CR dimension $\nu$ and dimension $n-d$, then the intrinsic complexification
of $X$ at every point is of dimension $n-d-\nu$.
\end{thm}

\begin{proof}[Proof of Theorem~\ref{thm:constsigma}]
It is sufficient to show that the dimension of the intrinsic complexification is
locally constant as $X$ is, in particular, connected.  Therefore, without loss
of generality, assume that $U$ is a good neighborhood for $X$ at some point
$q \in X$.  As $X$ is coherent, then $(\Sigma_p^U,p) = \Sigma_p$ for all
$p \in X$.  Hence, the dimension of $\Sigma_p^U$ equals the dimension of $\Sigma_p$.
Application of Lemma~\ref{lemma:constsigma} finishes the proof.
\end{proof}

Next, we consider real-algebraic subvarieties, that is, subvarieties
given by a polynomial.  Coherence is somewhat of a tricky subject for
real-algebraic subvarieties; a real-algebraic subvariety $X$ being coherent
as a real-analytic subvariety does not necessarily imply that the sheaf of
germs vanishing on $X$ at each point is generated by polynomials.
Therefore, we will say that a real subvariety $X$ has
\emph{polynomial defining functions} if there exists a single set of polynomials
$r_1,\ldots,r_k$ whose zero set is $X$ and whose germs define
$I_p(X)$ for every point $p \in X$.

\begin{thm} \label{thm:alhintcplx}
If $X \subset \C^n$ is an irreducible real-algebraic subvariety
with polynomial defining functions,
then the intrinsic complexification at every point $p \in X$ has the same
dimension.

In particular, if near some point, $X$ is a real-analytic CR
submanifold of CR dimension $\nu$ and dimension $n-d$, then the intrinsic complexification
of $X$ at every point is of dimension $n-d-\nu$.
\end{thm}

\begin{proof}[Proof of Theorem~\ref{thm:alhintcplx}]
Let $\sX \subset \C^n \times \C^n$ be the algebraic complex variety given by the complexified defining polynomials.
The set $\pi_1(\sX)$ is constructible by the Chevalley theorem (see e.g. \cite[pg. 395]{MR1131081}), and so it is
contained in a complex subvariety $Y \subset \C^n$ of the same dimension.
As $X$ is irreducible, so is $\sX$ and therefore so is $Y$.

Consider $p \in X$ and a good neighborhood $U$ of $p$ for $X$.
Then $\sX^U = \sX \cap (U \cap U^*)$.  Then $\pi_1(\sX^U)$ has the same dimension
as $\pi_1(\sX)$ and therefore the same dimension as $Y$.  As $Y$ is irreducible it
has the same dimension at all points.  Therefore, the intrinsic complexification
also has the same dimension at all points.

The second conclusion of the theorem follows then by standard CR results.
\end{proof}

We remark that the proof does not mean that the intrisic complexification
at a point is always equal to the germ of $\pi_1(\sX)$ at $p$ as that could
possibly contain more components, however, at most points these two germs
are equal.

\begin{thm}
If $U \subset \C^n$ is a domain and
$X \subset U$ is an irreducible real-analytic subvariety.  Let $X_{top}$ be the
top dimensional points.
If the Segre variety germ $\Sigma_p$ is of constant dimension on $X_{top}$,
then the dimension of the intrinsic complexification of $X$ is
constant on $X_{top}$.
\end{thm}

\begin{proof}
Fix a point $q \in X_{top}$ and consider a good neighborhood
$U'$ of $X$ at $p$. 
The hypotheses imply that $\Sigma_p^{U'}$ is
of constant dimension for $p \in U'$, and the same argument
as that in 
Lemma \ref{lemma:constsigma} finishes the proof.
\end{proof}

\begin{remark}
If $X \subset \C^n$ is a coherent real-analytic subvariety, then an application of \cite[Theorem 1.2]{MR4407043} (the set of Segre-degenerate points of a coherent subvariety is a subvariety)
and Theorem~\ref{thm:constsigma} shows that the subset 
 where the intrinsic complexification
of $X$ is larger than the generic value is contained in a real-analytic subvariety $S\subset X$.
A consequence of
Example~\ref{example:nonalgex} says that the dimension of the Segre variety and the dimension
of the intrinsic complexification are not necessarily related near Segre-degenerate points, and hence
we expect the stratification via dimensions of the Segre variety and the dimensions of the
intrinsic complexification to be different in general.
\end{remark}

\begin{example}
For noncoherent varieties the situation is yet more complicated.
The subvariety in $\C^2$ given by
\begin{equation*}
w^3\bar{z}^3+3w^2\bar{w}z\bar{z}^2+3w\bar{w}^2z^2\bar{z} -8w^2\bar{w}^2z\bar{z}+\bar{w}^3z^3=0,
\end{equation*}
see \cite[Example 2.6]{MR4407043}.  This irreducible (noncoherent) subvariety is a hypersurface
and a complex cone, with a so-called umbrella handle $\{ z=0,w\not= 0 \}$
which are regular points of real dimension 2.
The
Segre variety at the origin is of dimension 2, the Segre variety at the regular points of
hypersurface dimension is of dimension 1, and at the points of the umbrella handle
the variety is a complex line and hence the germ of the Segre variety at those points
is 1 dimensional, and so is the intrinsic complexification, while the intrinsic complexification
at all other points is 2 dimensional.  In particular, neither the set of points where
the intrinsic complexification is of dimension 2 nor the set of points where it is of dimension 1
is a subvariety.
\end{example}

\section{Intrinsic complexification under finite maps}

The Segre varieties, the intrinsic complexification,
and the precomplexification are preserved under
finite maps, even though an image of a subvariety might not be
a subvariety.  E.g., take $X = \R \subset \C$ and consider the map
$F(z) = z^2$.  Everything will be done with respect to $U=\C$.
Then $F(X)$ is the positive real ray, and hence not a
subvariety.  However, if we take instead of the image
the smallest real subvariety containing $F(X)$, that is $S=\R$,
then we find everything preserved.  The precomplexification $\sX$ of $X$
is given by $z = \zeta$ in $\C^2$.  The map $F$ complexifies to
the map $\sF(z,\zeta) = (z^2,\zeta^2)$, which takes $\sX$
to $\sX$.  We find that $\sX \cap \iota(\C) = X$ and 
$\sF(\sX) \cap \iota(\C) = \sX \cap \iota(\C) = S$.
The intrinsic complexification of both $X$ and $S$ is $\C$,
and $F$ preserves $\C$.

\begin{prop}
Let $(X,0) \subset (\C^n,0)$ be a germ of a real-analytic subvariety
of real dimension $\ell$, and
let $(F,0) \colon (\C^n,0) \to (\C^n,0)$ be a germ of a
finite
holomorphic mapping.  Complexify $(F,0)$ to $(\sF,0)$
by considering $\sF = F \oplus \overline{F(\bar{\cdot})}$.
Let $(\sX,0)$ be the complexification of $(X,0)$.
The germ $(F(X),0)$ is a germ of a semianalytic set of real
dimension $\ell$.  Let $(S,0)$ be the smallest germ of a
real-analytic subvariety containing $(F(X),0)$.
Then
\begin{enumerate}[(i)]
\item
$(S,0)$ is of real dimension $\ell$.
\item
$(\iota(S),0)= (\sF(\sX)\cap\iota(\C^n),0)$
\item
$(\sF(\sX),0)$ is the complexification of $(S,0)$.
\item
If $(Y,0)$ is the intrinsic complexification of $(X,0)$,
then $(F(Y),0)$ is the intrinsic complexification of $(S,0)$.
\end{enumerate}
\end{prop}

\begin{proof}
Denote by $(\Gamma_F,0)$ the germ of the graph of $F$ at $0\in \C^n_z \times \C^n_w$. Then
$F((X,0)) = \pi_w \left((\Gamma_F, 0)\cap \left(X \times \C^n,0\right) \right) $ is semianalytic
by the {\L}ojasiewicz-Tarski-Seidenberg theorem (see \cite[Lemma 5.3]{MR4407043}) of the same
dimension as $X$, therefore $S$ is of the same dimension too. 

For the second assertion, let $\rho$ be any germ of a real-analytic function near $0$.
We need to show that $\rho(F(z),\overline{F(z)})=0$ for $z\in X$ happens if and
only if $\rho(\sF (z,\zeta)) =0$ for $(z,\zeta)\in \sX$, which is the definition of the 
complexification. 

In order to prove the third assertion, assume without loss of generality that $(X,0)$ is 
irreducible. Hence both $(\sX,0)$ and $(\sF (\sX),0)$ are irreducible, and the assertion follows from (ii). 

The fourth assertion follows since $(F(Y),0)$ is a germ of a complex-analytic subvariety containing
$(S,0)$, and also coincides with the projection of the complexification of $(S,0)$.
\end{proof}

\begin{example}
    Coherence is not preserved under a finite map.  Let $X = \R^2 \subset \C_{(z,w)}^2$
    and let $F(z,w) = (z+izw,w^2)$.  Let $Y \subset \C_{(\zeta,\omega)}^2$ be given by
    $y^2-x^2s=0, t=0$.  The subvariety $Y$ is the standard Whitney umbrella (crosscap), which is a 
    noncoherent subvariety.
    The map $F$ is finite, and the image $F(X)$ consists of the points of $Y$ of dimension $2$.
\end{example}


\def\MR#1{\relax\ifhmode\unskip\spacefactor3000 \space\fi%
  \href{http://www.ams.org/mathscinet-getitem?mr=#1}{MR#1}}



\begin{bibdiv}
\begin{biblist}

\bib{BER:book}{book}{
  author={Baouendi, M. Salah},
  author={Ebenfelt, Peter},
  author={Rothschild, Linda Preiss},
  title={Real submanifolds in complex space and their mappings},
  series={Princeton Mathematical Series},
  volume={47},
  publisher={Princeton University Press, Princeton, NJ},
  date={1999},
  pages={xii+404},
  isbn={0-691-00498-6},
  review={\MR{1668103}},
}

\bib{MR1704996}{article}{
   author={Burns, Daniel},
   author={Gong, Xianghong},
   title={Singular Levi-flat real analytic hypersurfaces},
   journal={Amer. J. Math.},
   volume={121},
   date={1999},
   number={1},
   pages={23--53},
   issn={0002-9327},
   review={\MR{1704996}},
}

\bib{MR1224231}{book}{
   author={D'Angelo, John P.},
   title={Several complex variables and the geometry of real hypersurfaces},
   series={Studies in Advanced Mathematics},
   publisher={CRC Press, Boca Raton, FL},
   date={1993},
   pages={xiv+272},
   isbn={0-8493-8272-6},
   review={\MR{1224231}},
}

\bib{MR0477153}{article}{
   author={Diederich, Klas},
   author={Fornaess, John E.},
   title={Pseudoconvex domains with real-analytic boundary},
   journal={Ann. of Math. (2)},
   volume={107},
   date={1978},
   number={2},
   pages={371--384},
   issn={0003-486X},
   review={\MR{0477153}},
   doi={10.2307/1971120},
}

\bib{MR4311600}{article}{
   author={Lamel, Bernhard},
   author={Lebl, Ji\v r\'i},
   title={Segre nondegenerate totally real subvarieties},
   journal={Math. Z.},
   volume={299},
   date={2021},
   number={1-2},
   pages={163--181},
   issn={0025-5874},
   review={\MR{4311600}},
   doi={10.1007/s00209-020-02659-6},
}

\bib{MR1013362}{book}{
   author={Guaraldo, Francesco},
   author={Macr\`i, Patrizia},
   author={Tancredi, Alessandro},
   title={Topics on real analytic spaces},
   series={Advanced Lectures in Mathematics},
   publisher={Friedr. Vieweg \& Sohn, Braunschweig},
   date={1986},
   pages={x+163},
   isbn={3-528-08963-6},
   review={\MR{1013362}},
   doi={10.1007/978-3-322-84243-5},
}

\bib{MR4407043}{article}{
   author={Lebl, Ji\v r\'i},
   title={Segre-degenerate points form a semianalytic set},
   journal={Proc. Amer. Math. Soc. Ser. B},
   volume={9},
   date={2022},
   pages={159--173},
   review={\MR{4407043}},
   doi={10.1090/bproc/99},
}

\bib{MR1131081}{book}{
   author={\L ojasiewicz, Stanis\l aw},
   title={Introduction to complex analytic geometry},
   note={Translated from the Polish by Maciej Klimek},
   publisher={Birkh\"auser Verlag, Basel},
   date={1991},
   pages={xiv+523},
   isbn={3-7643-1935-6},
   review={\MR{1131081}},
   doi={10.1007/978-3-0348-7617-9},
}

\bib{Whitney:book}{book}{
   author={Whitney, Hassler},
   title={Complex analytic varieties},
   publisher={Addison-Wesley Publishing Co., Reading, Mass.-London-Don
   Mills, Ont.},
   date={1972},
   pages={xii+399},
   review={\MR{0387634}},
}

\end{biblist}
\end{bibdiv}
\end{document}